\documentclass[11pt, reqno]{amsart}
\usepackage{amssymb}
\usepackage{amsmath}
\usepackage{amsthm}
\usepackage{verbatim}
\usepackage[all]{xy}
\usepackage{url}
\usepackage{mathtools}
\usepackage{dsfont}
\usepackage{pdflscape}

\usepackage{listings}

\DeclareFontEncoding{OT2}{}{} 
\DeclareTextFontCommand{\textcyr}{\fontencoding{OT2}
    \fontfamily{wncyr}\fontseries{m}\fontshape{n}\selectfont}

\newcommand{\Ch}{\!\!\!\textcyr{Ch}}

\DeclareSymbolFont{rsfs}{U}{rsfs}{m}{n}
\DeclareSymbolFontAlphabet{\mathrsfs}{rsfs}

\usepackage{color}

\usepackage[dvipsnames]{xcolor}

\theoremstyle{plain}
\newtheorem{theorem}{Theorem}[section]

\newtheorem{proposition}[theorem]{Proposition}
\newtheorem{lemma}[theorem]{Lemma}
\newtheorem{corollary}[theorem]{Corollary}

\theoremstyle{definition}

\newtheorem{definition}[theorem]{Definition}

\newtheorem{subsec}[theorem]{}

\newcommand{\F}{{\mathds F}}
\newcommand{\Z}{{\mathds Z}}

\newcommand{\GG}{{\mathds G}}

\newcommand{\ov}{\overline}

\newcommand{\lra}{\longrightarrow}

\DeclareMathOperator{\coker}{coker}

\DeclareMathOperator{\Gal}{Gal}

\DeclareMathOperator{\Hom}{Hom}

\DeclareMathOperator{\im}{im}

\newcommand{\sss}{{\rm ss}}
\newcommand{\ssc}{{\rm sc}}

\newcommand{\into}{\hookrightarrow}
\newcommand{\onto}{\twoheadrightarrow}

\newcommand{\labelt}[1]{\xrightarrow{\makebox[1.2em]{\scriptsize ${#1}$}}}

\newcommand{\labeltoooo}[1]{\xrightarrow{\makebox[3.5em]{\scriptsize ${#1}$}}}

\newcommand{\longisoto}{{\ \labelt{\raisebox{-1.ex}{$\sim$}}\ }}

\newcommand{\isoto}{\longisoto}

\newcommand{\hs}{\kern 0.8pt}
\newcommand{\hssh}{\kern 1.2pt}
\newcommand{\hshs}{\kern 1.6pt}
\newcommand{\hssss}{\kern 2.0pt}

\newcommand{\hm}{\kern -0.8pt}
\newcommand{\hmm}{\kern -1.2pt}

\newcommand{\emm}{\bfseries}

\newcommand{\X}{{{\sf X}}}

\newcommand{\V}{\mathcal{V}}

\newcommand{\tor}{{\rm tor}}
\newcommand{\mm}{{\rm m}}

\newcommand{\loc}{{\rm loc}}
\newcommand{\alg}{{\rm alg}}

\newcommand{\GmL}{{{\mathds G}_{\mm,L}}}

\newcommand{\G}{\Gamma}

\newcommand{\tors}{{\rm Tors}}

\newcommand{\Gt}{{\G\hm,\hs\tors}}

\newcommand{\D}{{\Delta}}

\def\Tors{{\rm Tors}}
\def\bv{{\breve{v}}}

\def\Gbvt{{\G\hm_\bv\hs,\Tors}}

\def\sk{\smallskip}

\renewcommand{\varpi}{\pi}

\newcommand{\pia}{{\pi_1^\alg}}

\def\nc{{\rm nc}}

\def\Gm{{\GG_{\rm m}}}
\newcommand{\Y}{Y}

\newcommand{\Ups}{{\Upsilon\!}}

\def\Gbv{{\G\hm_{w}}}
\def\bv{w}
\def\D{{\Delta}}

\begin{document}

\title[Defect of weak approximation]
{The defect of weak approximation\\ for a reductive group\\ over a global field}

\author{Mikhail Borovoi\\
{\tiny with an appendix by}\\
Jean-Louis Colliot-Th\'el\`ene}

\address{Raymond and Beverly Sackler School of Mathematical Sciences,
Tel Aviv University,
Ramat Aviv,
 6997801 Tel Aviv, Israel}
\email{borovoi@tauex.tau.ac.il}

\address{Laboratoire de Math\'ematiques d'Orsay, Universit\'e Paris-Saclay, CNRS, Orsay, France}
\email{jean-louis.colliot-thelene@universite-paris-saclay.fr}

\thanks{The author was partially supported
by the Israel Science Foundation (grant 1030/22).}



\date{\today}

\begin{abstract}
We compute the defect of weak approximation for a reductive group $G$ over a global field $K$
in terms of the algebraic fundamental group of $G$.
\end{abstract}

\maketitle

\tableofcontents

\section{Introduction}
\label{s:Intro}

\begin{subsec}
Let $K$ be a global field (a number field or a global function field), and let $K^s$ be a fixed separable closure of $K$.
Let $G$ be a  reductive group over $K$
(we follow the convention of SGA3, where reductive groups are assumed to be connected).
Let $\V_K$  denote the set of places of $K$, and
let $S\subset\V_K$ be a finite set of places.
Consider the group
\[ G(K_S)\coloneqq \prod_{v\in S} G(K_v)\]
where $K_v$ denotes the completion of $K$ at $v$.
The group $G(K)$ embeds diagonally into $G(K_S)$.
One says that $G$ has the {\em weak approximation property in $S$}  (for short (WA$_S$)\hs)  if $G(K)$ is dense in $G(K_S)$.
One says that $G$ has the {\em weak approximation property} if it has the weak approximation property in $S$
for any finite subset $S\subset\V_K$.

Without assuming that $G$ has the weak approximation property in $S$,
let $\ov{G(K)}_S$ denote the closure of $G(K)$ in $G(K_S)$.
Sansuc \cite[\S\hs3]{Sansuc} showed (in the number field case)
that the subgroup $\ov{G(K)}_S$ is normal in $G(K_S)$ and that the quotient group
\[A_S(G) \coloneqq G(K_S)/\ov{G(K)}_S\]
is a finite abelian group.
Sansuc's argument extends to the function field case.
We say that $A_S(G)$ is the {\em defect of weak approximation for $G$ in $S$}.
\end{subsec}

\begin{subsec}
In the number field case, Sansuc computed $A_S(G)$  when $G$ is semisimple,
or, more generally, when $G$ admits a {\em special covering},
that is,  $G$  fits into a short exact sequence of special kind
\begin{equation}\label{e:special}
1\to B\to G'\to G\to 1.
\end{equation}
Here $G'$ is the product  of a simply connected semisimple $K$-group and a quasi-trivial $K$-torus,
and $B$ is a finite abelian $K$-group.
Namely, Sansuc constructed an isomorphism
\begin{equation}\label{e:special-AS(G)}
A_S(G)\isoto \Ch^1_S(B)\coloneqq \coker\Big[H^1(K,B)\to\prod_{v\in S} H^1(K_v,B)\Big].
\end{equation}
Note that there exist reductive $K$-groups not admitting a special covering.
\end{subsec}

\begin{subsec}\label{ss:CT}
In the number field case, Colliot-Th\'el\`ene \cite[Theorem 9.4(i)]{CT-RF}
computed the finite group $A(G)\coloneqq\varprojlim_S A_S(G)$
for any reductive $K$-group $G$.
He considered a {\em flasque resolution} of $G$, that is,  a short exact sequence \eqref{e:special}
in which $G'$ is a {\em quasi-trivial} reductive $K$-group
(see \cite[\S\hs2]{CT-RF} or Definition \ref{d:qt} below)
and $B$ is a {\em flasque} $K$-torus
(\hs see \cite[0.8]{CT-RF}\hs).
Colliot-Th\'el\`ene computed $A(G)$ by a formula similar to \eqref{e:special-AS(G)}.
\end{subsec}

\begin{subsec}
In this paper,
we compute the defect of weak approximation  $A_S(G)$  for a reductive $K$-group $G$
in terms of the algebraic fundamental group  $\pia(G)$
introduced in \cite[Section 1]{Borovoi-Memoir}
(and also by Merkurjev \cite{Merkurjev} and Colliot-Th\'el\`ene \cite{CT-RF}).
Let $G^\sss=[G,G]$ denote the derived group of $G$, which is semisimple, and let $G^\ssc$
denote the universal cover of  $G^\sss$, which is simply connected;
see \cite[Proposition (2:24)(ii)]{Borel-Tits-C} or \cite[Corollary A.4.11]{CGP}.
Consider the composite homomorphism
\[\rho\colon G^\ssc\onto G^\sss\into G,\]
which is in general is neither injective nor surjective.
For a maximal torus  $T\subseteq G$,
we denote
\[T^\ssc=\rho^{-1}(T)\subseteq G^\ssc.\]
Following \cite{Borovoi-Memoir}, we consider the {\em algebraic fundamental group of $G$}  defined by
\[\pia (G)=\X_*(T)/\rho_*\X_*(T^\ssc)\]
where $\X_*$ denotes the cocharacter group over $K^s$.
The absolute  Galois group $\Gal(K^s/K)$ naturally acts on $\pia(G)$,
and the Galois module $\pia(G)$ is well defined (does not depend on the choice of $T$
up to a transitive system of isomorphisms);
see \cite[Lemma 1.2]{Borovoi-Memoir}.
Note that when $G=T$ is a torus, we have $\pia(G)=\X_*(T)$.

Write $M=\pia(G)$.
Let $T\subseteq G$ be a maximal torus.
Choose a finite Galois extension $L/K$ in $K^s$ splitting $T$.
Set $\G=\Gal(L/K)$, the Galois group of $L$ over $K$.
Then the finite group  $\G$ naturally acts on $M$.
Moreover,  $\G$ naturally acts on $L$ and on its set of places $\V_L$.
We compute $A_S(G)$ in terms of the $\G$-module $M$ and the $\G$-set $\V_L$.
\end{subsec}

\begin{subsec}
Following Tate \cite{Tate}, we consider the group of finite formal linear combinations
\[ M[\V_L]=\Bigg\{\sum_{w\in\V_L} m_w\cdot w\ \big|\ m_w\in M\Bigg\}\]
and its subgroup
\[M[\V_L]_0=\Big\{\sum m_w\cdot w\in M[\V_L] \ \big|\ \sum m_w=0\Big\}. \]
Then $\G$ naturally acts on the groups $ M[\V_L]$ and $ M[\V_L]_0$.
We prove the following theorem.
\end{subsec}

\begin{theorem}[Theorem \ref{t:AS-H1}]
\label{t:AS-H1-Intro}
\[ A_S(G)\cong\coker\Big[H_1(\G, M[\V_L]_0)\to \bigoplus_{v\in S} H_1(\Gbv, M)\Big]\]
where $H_1$ denotes  group homology.
Here for $v\in \V_K$, we choose a place $\bv\in\V_L$ over $v$,
and we denote by $\Gbv$ the corresponding decomposition group
(the stabilizer of $\bv$ in $\G$).
\end{theorem}

Theorem \ref{t:AS-H1-Intro} describes $A_S(G)$ in terms of the finite groups $ H_1(\Gbv, M)$
and the infinite group $H_1(\G, M[\V_L]_0$.
Our main result is Theorem \ref{t:AS-H1-cor} in  Section \ref{s:main},
which describes  $A_S(G)$ in terms of  finite groups only.

The plan of the rest of the paper is as follows.
In Section \ref{s:QT},
for a reductive group $G$ over $K$ splitting over a finite Galois extension $L/K$,
we introduce the notion of an $L/K$-free resolution of $G$,
and we show that $G$ admits  an $L/K$-free resolution.

In Section \ref{s:Ch} we compute $\Ch^1_S(B)$ for a $K$-torus $B$
in terns of $\pia(B)=\X_*(B)$.
In Section \ref{s:main} we prove Theorems \ref{t:AS-H1} and \ref{t:AS-H1-cor},
which are our main results.
In Section \ref{s:example} we consider an example over a global function field $K$:
we compute $A_S(T)$ for a certain $K$-torus $T$.
In Appendix \ref{app:CT}, J.-L.~Colliot-Th\'el\`ene gives an alternative proof
of the existence of an $L/K$-free resolution of a reductive $K$-group
splitting over a finite Galois extension $L/K$.

The arXiv version of this paper contains Appendix B. 
We provide a listing of our Magma program computing $A_S(G)$.
Our input is the effective Galois group $\G$
acting on the algebraic fundamental group $M=\pia(G)$,
the list of  non-cyclic decomposition groups $\G\hm_w$ for $v\in S$,
and the list of  non-cyclic decomposition groups for $v\in S^\complement\coloneqq \V_K\smallsetminus S$.
We compute $A_S(G)$ using Theorem \ref{t:AS-H1-cor}.

\bigskip
\noindent
{\sc Acknowledgements:}\ \
The author is grateful to Dipendra Prasad for his suggestion
to compute the defect of weak approximation for a reductive group.
We thank  Boris Kunyavski\u{\i} for helpful discussions and  email correspondence.
Special thanks to Jean-Louis Colliot-Th\'el\`ene for writing Appendix A.
This paper was partly written during the author's visit to Max Planck Institute for Mathematics, Bonn,
and the author is grateful to this  institute for its  hospitality, support, and excellent working conditions.

\section{Quasi-trivial groups and $L/K$-free groups}
\label{s:QT}

\begin{subsec}
A torus $T$ over a field $K$ is called {\em quasi-trivial} if its cocharacter group
$\X_*(T)$ admits a $\Gal(K^s/K)$-stable basis.
Then by Shapiro's lemma and Hilbert's Theorem 90 we have $H^1(K,T)=1$.
If $K$ is a global field, then $T$ has the weak approximation property
(because a quasi-trivial torus  is a $K$-rational variety).
\end{subsec}

We need the following two known theorems.

\begin{theorem}\label{t:scss}
Let $G$ be a simply connected semisimple group over a global field $K$, and let
$\V_f(K)$ and $\V_\infty(K)$ denote the sets of finite (non-archimedean)
and infinite (archimedean) places of $K$, respectively.
Then:
\begin{enumerate}
\item[(\rm i)] For all $v\in\V_f(K)$ we have $H^1(K_v,G)=1$.

\item[(\rm ii)] The localization map
\[\eta\colon\,H^1(K,G)\to \prod_{v\in\V(K)} H^1(K_v,G)=\!\!\prod_{v\in\V_\infty(K)} H^1(K_v,G)\]	
is bijective.
\end{enumerate}
\end{theorem}

\begin{proof}
Assertion (\rm i) is a theorem of Kneser and of Bruhat and Tits;
see Platonov and Rapinchuk \cite[Theorem 6.4, p.~284]{PR} for the number field case,
and Bruhat and Tits \cite[Theorem 4.7(ii)]{BT} for  the general case.
Assertion (\rm ii) is the celebrated Hasse principle of Kneser, Harder, and Chernousov;
see Platonov and Rapinchuk \cite[Theorem 6.6, p.~286]{PR} for the number field case,
and Harder \cite{Harder} for the function field case.
\end{proof}

\begin{theorem}
Let $G$ be a simply connected semisimple group over a global field $K$.
Then $G$ has the weak approximation property, that is,
for any finite subset $S\subset \V(K)$, the group $G(K)$ is dense in $G(K_S)\coloneqq \prod_{v\in S}G(K_v)$.
\end{theorem}

Indeed, $G$ has the weak approximation property
because it has the strong approximation property.
This was proved  by  Platonov in characteristic 0
(see Platonov and Rapinchuk \cite[Theorem 7.12, p.~427]{PR}),
and by G. Prasad \cite{GPrasad} in positive characteristic.

\begin{definition}[{Colliot-Th\'el\`ene \cite[Definition 2.1]{CT-RF}}]
\label{d:qt}
A reductive group $G$ over a field $K$ is {\em quasi-trivial}
if its commutator subgroup $G^\sss$ is simply connected and the quotient torus
$G^\tor\coloneqq G/G^\tor$ is a quasi-trivial torus.	
In other words, $G$ is quasi-trivial if it fits into the exact sequence
\begin{equation}\label{e:qt}
1\to G^\ssc\to G\to G^\tor\to 1,
\end{equation}
where $G^\ssc$ is a simply connected semisimple group,
and $G^\tor$ is a quasi-trivial torus.	
\end{definition}

\begin{proposition}[{Colliot-Th\'el\`ene \cite[Proposition 9.2(i,iii)]{CT-RF} in the number field case}]
\label{p:HP-qt}
Let $G$ be quasi-trivial group over a global field $K$.
Then:
\begin{enumerate}
\item[(\rm i)] For any $v\in\V_f(K)$ we have $H^1(K_v,G)=1$.

\item[(\rm ii)] The localization map
\[\loc_\infty\colon\,H^1(K,G)\to \!\!\prod_{v\in\V_\infty(K)}\!\!\!\! H^1(K_v,G)\]	
is bijective.
\end{enumerate}
\end{proposition}

\begin{proof}
Concerning (\rm i), from \eqref{e:qt} we obtain a cohomology exact sequence
\[H^1(K_v,G^\ssc)\to H^1(K_v,G)\to H^1(K_v,G^\tor)\]
where $H^1(K_v,G^\tor)$ is trivial because $G^\tor$ is a quasi-trivial torus,
and $H^1(K_v,G^\ssc)$ is trivial by Theorem \ref{t:scss}(i).
Thus $H^1(K_v,G)$ is trivial.

Concerning (\rm ii),
for the case of a number field see  \cite[Proposition 9.2(iii)]{CT-RF}.
In the case of a function field,
 from \eqref{e:qt} we obtain a cohomology exact sequence
\[H^1(K,G^\ssc)\to H^1(K,G)\to H^1(K,G^\tor)\]
where $H^1(K,G^\tor)$ is trivial because $G^\tor$ is a quasi-trivial torus,
and $H^1(K,G^\ssc)$ is trivial in the function field case by Theorem \ref{t:scss}(ii).
Thus $H^1(K,G)$ is trivial, as desired.
\end{proof}

\begin{proposition}[{Colliot-Th\'el\`ene \cite[Proposition 9.2(ii)]{CT-RF} in the number field case}]
\label{p:WA-qt}
Let $G$ be a quasi-trivial group over a global field $K$.
Then $G$ has the weak approximation property, that is, for any finite subset $S\subset \V(K)$,
the group $G(K)$ is dense in $G(K_S)$.
\end{proposition}

\begin{proof}
The proof of \cite[Proposition 9.2(ii)]{CT-RF} for the group $A(G)$  in the number field case
immediately extends to $A_S(G)$ and to the function field case.
\end{proof}

\begin{definition}
Let $T$ be a torus over a field $K$, and let $L/K$ be a finite Galois extension.
We say that $T$ is {\em $L/K$-free} if $T$ splits over $L$ and the $\Z[\Gal(L/K)]$-module $\X_*(T)$ is free.
Alternatively, $T$ is $L/K$-free if
it is isomorphic to $(R_{L/K}\GmL)^{n_T}$ for some integer $n_T\ge 0$,
where $\GmL$ is the multiplicative group over $L$ and $R_{L/K}$ denotes the Weil restriction of scalars.
\end{definition}

Observe that any $L/K$-free torus is quasi-trivial.

\begin{definition}\label{d:L/K-free-group}
Let $G$ be a reductive $K$-group with derived group $G^\sss=[G,G]$.
Write $G^\tor=G/G^\sss$.
We say that $G$ is {\em $L/K$-free} if $G^\sss$ is simply connected and the $K$-torus $G^\tor$ is $L/K$-free.
\end{definition}

Observe that any $L/K$-free reductive $K$-group is quasi-trivial.

\begin{definition}\label{d:L/K-free}
Let $G$ be a reductive $K$-group. An $L/K$-free resolution of $G$ is a short exact sequence of $K$-groups
\begin{equation}\label{e:L/K-free}
 1\to B\to G'\to G\to 1
\end{equation}
where $G'$ is an $L/K$-free reductive $K$-group and $B$ is a $K$-torus.
\end{definition}

\begin{proposition}\label{p:E-L/K-free}
Any reductive $K$-group  $G$
that splits over a finite Galois extension $L/K$,
admits an $L/K$-free resolution.
\end{proposition}

\begin{proof}
We   follow  the proof of \cite[Proposition-Definition 3.1]{CT-RF}.
Since $G$ splits over $L$, it has a maximal torus $T$ that splits over $L$.
Let $Z$ denote the identity component of the center of $G$ (that is, the radical of $G$);
this is a $K$-torus splitting over $L$, because it is a subtorus of the torus $T$ splitting over $L$.
There exists  a surjective homomorphism of $K$-tori  $Q\to Z$  with $Q$  being $L/K$-free.
Let $T^\ssc\subset G^\ssc$ denote the preimage of $T$ in $G^\ssc$.
Then $T^\ssc$ is isogenous to the $K$-torus $T^\sss\coloneqq T\cap G^\sss$.
We see that both $T^\sss$ and $T^\ssc$  split over $L$.
We have a natural surjective $K$-homomorphism $G^\ssc\times Q\to G$.
Let $Z'$ denote the kernel of this homomorphism;
then  $Z'$ is the kernel of the surjective homomorphism $T^\ssc\times Q\to T$,
and therefore, $Z'$ is a group of multiplicative type over $K$ (not necessary smooth).
It follows that we have a short exact sequence of character groups
\[0\to\X^*(T)\to\X^*(T^\ssc)\oplus \X^*(Q)\to \X^*(Z')\to 0;\]
see Milne \cite[Theorem 12.9]{Milne-AG}.
We see that  the absolute Galois group  $\Gal(K^s/K)$
acts on  the character group $\X^*(Z')$ via $\Gal(L/K)$.

By Lemma \ref{l:G-free} below, there exists a short exact sequence
\[ 1\to Z'\to B\to F\to 1,\]
where  $B$ is a $K$-torus that splits over $L$,
and $F$ is an $L/K$-free $K$-torus.
The diagonal map defines an embedding of $Z'$ into the product
$(G^\ssc\times Q)\times B$ with central image.
Let $G'$ be the quotient of the reductive group $(G^\ssc\times Q)\times B$ by $Z'$.
Then $G'$ is a reductive $K$-group,
which fits into the following commutative diagram with exact rows and columns
(where we replace the given arrow $Z'\to B$ by its inverse):
\[
\xymatrix{
     &1\ar[d]   &1\ar[d] \\
1\ar[r]  &Z' \ar[r]\ar[d] &G^\ssc\times Q\ar[r]\ar[d]   &G\ar[r]\ar@{=}[d]  &1 \\
1 \ar[r]  &B\ar[r]\ar[d]  & G'\ar[r]\ar[d]                     &G \ar[r]                 &1 \\
     &F\ar@{=}[r]\ar[d]  &F\ar[d] \\
     &1   &1
}
\]
The quotient  of the group $G'$ by the  normal subgroup $G^\ssc\times 1\subset  G^\ssc\times Q\subset G'$
is a $K$-group extension of the $K$-torus $F$  by the $K$-torus $Q$.
By \cite[Section 0.7]{CT-RF} such a $K$-group is a $K$-torus.
Since any extension of an $L/K$-free torus by an $L/K$-free torus is a  split extension
(as one can see on the level of cocharacter groups),
we conclude that the quotient  of $G'$ by $G^\ssc$ is an $L/K$-free torus.
The derived group $G^{\prime\,\sss}$  can be identified with $G^\ssc$,
and the group $G^{\prime\,\tor}$ is an $L/K$-free torus.
Thus the reductive group $G'$ is $L/K$-free,
and the middle row of the diagram is a desired $L/K$-free resolution of $G$.
\end{proof}

See Appendix \ref{app:CT} for an alternative proof of Proposition \ref{p:E-L/K-free}.

\begin{lemma}
\label{l:G-free}
Let $\G$ be a finite group and $M$ be a finitely generated $\G$-module.
Then there exists a resolution
\begin{equation*}
0\to M^{-1}\to M^0\to M\to 0
\end{equation*}
where $M^0$ is a finitely generated $\Z$-torsion-free $\G$-module
and $M^{-1}$ is a free $\G$-module.
\end{lemma}

\begin{proof}
See \cite[Lemma 4.1.1]{BK}, or Milne and Shih \cite[Proof of Lemma 3.2]{MSh}.
\end{proof}

\begin{theorem}\label{t:FR}
Let $G$ be a reductive group over a global field $K$.
Consider a short exact sequence
\begin{equation}\label{e:QTR}
 1\to B\to G'\to G\to 1
\end{equation}
where $G'$ is a quasi-trivial $K$-group and $B\subset G'$ is a smooth central $K$-subgroup.
Let $S$ be a finite set of places of $K$.
Then the closure $\ov{G(K)}_S$ of $G(K)$ in $G(K_S)$ is a normal subgroup of finite index,
and the connecting homomorphism $G(K_S)\to H^1(K_S,B)$ induces an isomorphism
\begin{equation}\label{e:AS-ChB}
A_S(G)\isoto\Ch^1_S(B).
\end{equation}
\end{theorem}

\begin{proof}
Sansuc \cite[Theorem 3.3]{Sansuc} considers a resolution \eqref{e:QTR}
in the case when $K$ is a number field and $B$ is a {\em finite abelian} $K$-group.
Moreover, he assumes that the resolution \eqref{e:QTR} is a special covering.
In this case, Sansuc shows that \eqref{e:QTR} induces an isomorphism \eqref{e:AS-ChB}.
Moreover, Colliot-Th\'el\`ene \cite[Theorem 9.4(i)]{CT-RF}
considers a  flasque resolution \eqref{e:QTR} of $G$, when $K$ is a number field.
Then $B$ is a flasque torus.
Colliot-Th\'el\`ene shows that a flasque resolution \eqref{e:QTR} induces an isomorphism \eqref{e:AS-ChB}.

Our Theorem \ref{t:FR} generalizes \cite[Theorem 3.3]{Sansuc} and \cite[Theorem 9.4(i)]{CT-RF}.
The proof of Sansuc \cite{Sansuc} (inspired by Kneser \cite{Kneser-Schwache})
immediately generalizes to our case.
\end{proof}

\section{Computing $\Ch^1_S(B)$ for a torus $B$}
\label{s:Ch}

Let $B$ be a torus over a global field $K$ with cocharacter group $\Y=\X_*(B)$.
Let $L/K$ be a finite Galois extension splitting $T$. We write $\G=\Gal(L/K)$.
In this section we compute $\Ch^1_S(B)$ for a finite set $S$ of places of $K$ in terms of the $\G$-module $Y$.
We need two lemmas.

\begin{lemma}[probably known]
\label{l:torsion-free}
Let $\D$ be a finite group and $A$ be a $\D$-module. If $A$ is torsion-free, then
\[ H^{-1}(\D,A)=A_{\D,\Tors}\]
where $A_\D$ denotes the group of coinvariants of $\D$ in $A$,
and $A_{\D,\Tors}\coloneqq(A_\D)_\Tors$, the torsion subgroup of $A_\D$.
\end{lemma}

\begin{proof}
By definition,
\[H^{-1}(\D,A)=\ker\big[{\rm Nm}\colon A_\D\to A^\D\big]\]
where $A^\D$ denote the group of invariants of $\D$ in $A$;
see \cite[Section IV.6]{CF}.
The group $H^{-1}(\D, A)$ is killed by $\#\D$ (see \cite[Section IV.6, Corollary 1 of Proposition 8]{CF}\hs),
whence $H^{-1}(\D,A)\subseteq A_{\D,\Tors}$\hs.
Since $A$ is torsion-free, so is $A^\D$,
whence
\[H^{-1}(\D,A)=\ker{\rm Nm}\supseteq A_{\D,\Tors}\hs.\]
Thus $H^{-1}(\D,A)=A_{\D,\Tors}$\hs.
\end{proof}

\begin{lemma}\label{l:torus}
With the notation of the beginning of this section,
for any place $v$ of $K$ choose a place $w$ of $L$ over $v$,
and denote by $\Gbv\subseteq \G$ the corresponding decomposition group (the stabilizer of $\bv$ in $\G$).
Then we have  canonical isomorphisms
\[H^1(K_v,B)\isoto Y_\Gbvt\hs,\quad\  H^1(K,B)\isoto \big(Y[\V_L]_0\big)_\Gt\hs.\]
\end{lemma}

\begin{proof}
Since the torus $B$ splits over $L_w$\hs,
we have $H^1(K_v,B)=H^1(L_\bv/K_v,B)$; see Sansuc \cite[(1.9.2)]{Sansuc}.
We have the Tate-Nakayama isomorphism
\[H^{-1}(\Gbv,Y)\isoto H^1(\Gbv, Y\otimes_\Z L_\bv^\times)=H^1(L_w/K_v,B);\]
see  Tate \cite[Theorem on page 717]{Tate}.
By Lemma \ref{l:torsion-free}, we have
$H^{-1}(\Gbv,Y)=Y_\Gbvt$\hs, whence
$H^1(L_w/K_v,B)\cong Y_\Gbvt$\hs.

Similarly, we have $H^1(K,B)=H^1(L/K,B)$; see \cite[(1.9.2)]{Sansuc}.
We have the Tate isomorphism
\[H^{-1}\big(\G,Y[\V_L]_0\big)\isoto H^1(L/K,B);\]
see \cite[Theorem on page 717]{Tate}.
Since $Y[\V_L]_0$ is torsion-free, we conclude by Lemma \ref{l:torsion-free} that
\[ H^1(L/K,B)\cong H^{-1}\big(\G,Y[\V_L]_0\big)=\big(Y[\V_L]_0\big)_\Gt\hs,\]
as desired.
\end{proof}

\begin{subsec}
With the notation of Lemma \ref{l:torus},  consider the groups
\begin{equation}\label{e:Ups}
\Ups_S= \bigoplus_{v\in S}\Y_\Gbvt\quad\ \text{and}\quad\ \Ups_{S^\complement}
      = \bigoplus_{v\in S^\complement}\!\Y_\Gbvt\hs.
\end{equation}
By Lemma \ref{l:torus} we have canonical isomorphisms
\begin{equation*}
\Ups_S= \bigoplus_{v\in S}H^1(K_v,B)\quad\ \text{and}\quad\ \Ups_{S^\complement}
      = \bigoplus_{v\in S^\complement}\!H^1(K_v,B).
\end{equation*}
For each $v\in\V_K$\hs,  consider the natural projection homomorphism
\[\pi_v\colon \Y_\Gbvt\to Y_\Gt\hs.\]
We have a natural homomorphism
\[\pi_S\colon\, \Ups_S\to \Y_\Gt\hs,  \ \,\big[ y_v\big]_{v\in S}\mapsto \Bigg[\,\sum_{v\in S}\pi_v(y_v)\,\Bigg].\]
Similarly, we have a natural homomorphism
\[  \pi_{S^\complement}\colon\  \Ups_{S^\complement}\to \Y_\Gt\hs.\]
\end{subsec}

\begin{theorem}
\label{p:Ch-complement}
There are canonical isomorphisms
\[ \Ch^1_S(B)\,\cong\Ups_S\hs/\hs \pi_S^{-1}(\im\pi_{S^\complement})
\cong \im\pi_S/(\im\pi_S\cap\im\pi_{S^\complement}).\]
\end{theorem}

\begin{proof}
Consider the following commutative  diagram:
\begin{equation}\label{e:diagram}
\begin{aligned}
\xymatrix@C=14mm{
H^1(K,B)\ar[d]\ar[r]^-\sim &  H^{-1} \big(\Y[\V_L]_0\big)\ar[r]^-\sim\ar[d]                            &   \big(\Y[\V_L]_0\big)_\Gt\hskip-7mm\ar[d] \\
H^1(K_S,B)\ar[r]^-\sim     &\bigoplus\limits_{v\in S}  H^{-1}(\G\hm_{w}, Y) \ar[r]^-\sim      & \bigoplus\limits_{v\in S}\Y_\Gbvt\hskip-4mm
}
\end{aligned}
\end{equation}
in which the rectangle at left
comes from Tate \cite[Theorem on page 717]{Tate}.
From \eqref{e:diagram} we obtain a canonical isomorphism
\[\Ch^1_S(B)=\coker\big[H^1(K,B)\to H^1(K_S,B) \big]\isoto \coker\Big[ \big(\Y[\V_L]_0\big)_\Gt\hm\to \Ups_S\Big]\]
where $\Ups_S$ is as in \eqref{e:Ups}.
From the short exact sequence
\[ 0\to \Y[\V_L]_0\to \Y[\V_L]\to Y\to 0\]
we obtain an exact sequence
\[H^{-1}\big(\G, \Y[\V_L]_0\big)\to\bigoplus\limits_{v\in \V_K}
     H^{-1}(\G\hm_w,Y)\to H^{-1}(\G, \Y),\]
which by Lemma \ref{l:torsion-free}  gives an exact sequence
\[ \big(\Y[\V_L]_0\big)_\Gt\ \lra\ \Ups_S\oplus\Ups_{S^\complement}
    \ \labeltoooo{\pi_S +\pi_{S^\complement}}  \Y_\Gt\hs.\]
We see that
\[ \im \Big[ \big(\Y[\V_L]_0\big)_\Gt\hs\to\hs \Ups_S\Big]\,
     =\,\pi_S^{-1}(\im\pi_{S^\complement}),\]
which gives the first isomorphism of the theorem. The second isomorphism is obvious.
\end{proof}

\begin{corollary}\label{c:Sigma}
Let $S'\subset S\subset \V_K$ be two finite  sets of places of $K$
such that for each place $v\in \Sigma\coloneqq S\smallsetminus S'$ and $w\in\V_L$ over $v$,
there exists a place $v^\complement \in S^\complement$
and $w^\complement\in \V_L$ over $v^\complement$
with $\G\hm_{w^\complement}=\G\hm_w$.
Then the natural epimorphism $\Ch^1_S(B)\to \Ch^1_{S'}(B)$ is an isomorphism.
\end{corollary}

\begin{proof}
By Theorem \ref{p:Ch-complement}, it suffices to show that for any $v\in \Sigma$,
the direct summand  $\Y_\Gbvt$  of $\Ups_S $ is contained in
$ \pi_S^{-1}(\im\pi_{S^\complement})$.
This follows from the equality $\G\hm_{w^\complement}=\G\hm_w$.
\end{proof}

\begin{corollary}[Sansuc \cite{Sansuc}]
\label{c:Sansuc-Ch}
Let $S_{\rm nc}\subseteq S$ (resp.  $S_{\rm c}\subseteq S$)
denote the subset of places with non-cyclic (resp., cyclic)
decomposition groups in $\G=\Gal(L/K)$.
Then the natural epimorphism $\Ch^1_S(B)\to \Ch^1_{S_{\rm nc}}(B)$ is an isomorphism.
\end{corollary}

\begin{proof}
Let $v\in S_{\rm c}$ and $w$ be a place of $L$ over $v$.
Since the decomposition group $\G\hm_\bv\subset\G$ is cyclic and the set $S$ is finite,
by the Chebotarev density theorem,
see, for instance, Neukirch \cite[\hs Chapter VII, Theorem (13.4)\hs]{Neukirch},
there exist $v^\complement\in S^\complement$ and $\bv^\complement\in\V_L$ over $v^\complement$
with decomposition group $\G\hm_{\bv^\complement}=\G\hm_\bv$,
and we conclude by Corollary \ref{c:Sigma}.
\end{proof}

\section{$A_S(G)$ in terms of $\X_*(B)$ and  in terms of $\pi_1^\alg(G)$}
\label{s:main}

\begin{subsec}\label{ss:MLG}
Let $G$ be a reductive group over a global field $K$.
Write $M=\pia(G)$.
Let $T\subseteq G$ be a maximal torus,
and let $L/K$ be a finite Galois extension in $K^s$ splitting $T$;
then $\G\coloneqq \Gal(L/K)$ naturally acts on $M$.

Choose  an $L/K$-free resolution \eqref{e:L/K-free} as in Definition \ref{d:L/K-free},
 where $B$ is a $K$-torus, and  consider the sequence
of $\Gal(L/K)$-modules
\begin{equation}\label{e:Y-M'-M}
0\to \Y\to M'\to M\to 0
\end{equation}
where
\[\Y=\pia(B)=\X_*(B)\quad\ \text{and}\quad\   M'=\pia(G').\]
This sequence is exact; see \cite[Proposition 6.8]{CT-RF} or \cite[Lemma  6.2.4]{BK}.
Since $\G$ acts on $\Y$ and on $L$, it naturally acts on $\V_L$, on $\Y[\V_L]$, and on $\Y[\V_L]_0$\hs.
\end{subsec}

\begin{theorem}\label{t:AS-0}
For  an $L/K$-free resolution \eqref{e:L/K-free} and a finite set of places $S\subset \V_K$\hs, there is a canonical isomorphism
\[A_S(G)\,\isoto\,\coker\Big[\big(\Y[\V_L]_0\big)_\Gt\to \bigoplus_{v\in S} \Y_\Gbvt\Big].  \]
\end{theorem}

\begin{proof}
From diagram \eqref{e:diagram} we obtain that
\begin{align*}
\Ch^1_S(B)\coloneqq \coker\Big[H^1(K,B)\to H^1(K_S,B)\Big]\cong
                    \coker\Big[\big(\Y[\V_L]_0\big)_\Gt\to\bigoplus\limits_{v\in S}\Y_\Gbvt \Big]&,
\end{align*}
and the theorem follows from Theorem \ref{t:FR} giving an isomorphism
$A_S(G)\isoto \Ch^1_S(B)$.
\end{proof}

\begin{theorem}\label{t:AS-0-cor}
For an $L/K$-free resolution \eqref{e:L/K-free} and a finite set of places $S\subset \V_K$\hs,
there is a canonical isomorphism
\[A_S(G)\,\isoto\,
   \im \pi_{S_\nc}\hs /\big(\im\pi_{S_\nc}\cap\im \pi_{S^\complement}\hs\big).\]
\end{theorem}

\begin{proof}
By Theorem \ref{t:FR} we have a canonical isomorphism $A_S(G)\isoto\Ch^1_S(B)$,
and by Corollary \ref{c:Sansuc-Ch} we have a canonical isomorphism
$\Ch^1_S(B)\isoto\Ch^1_{S_\nc}(B)$.
By Theorem \ref{p:Ch-complement} we have a canonical isomorphism
\[ \Ch^1_{S_\nc}(B)\isoto
\im \pi_{S_\nc} /\big(\im\pi_{S_\nc}\cap\im \pi_{(S_\nc)^\complement}\hs\big).\]
By the Chebotarev density theorem, we have
\[\im \pi_{(S_\nc)^\complement}=\im \pi_{S^\complement}.\]
Thus we obtain a desired canonical isomorphism of the theorem.
\end{proof}

\begin{corollary}[Sansuc \cite{Sansuc}]
Assume that $G$ splits over a finite Galois extension $L/K$, and
let $S_{\rm nc}\subseteq S$  denote the subset of places with non-cyclic
decomposition group in $\G=\Gal(L/K)$.
Then the natural epimorphism $A_S(G)\to A_{S_{\rm nc}}(G)$ is an isomorphism.
\end{corollary}

\begin{proof}
For an $L/K$-free resolution \eqref{e:L/K-free} as in Definition \ref{d:L/K-free},
 where $B$ is a $K$-torus,
consider the commutative diagram
\[
\xymatrix{
A_S(G)\ar[r]^-\sim\ar[d]           &\Ch^1_S(B)\ar[d]^-\sim\\
A_{S_{\rm nc}}(G)\ar[r]^-\sim   & \Ch^1_{S_{\rm nc}}(B)
}
\]
in which the horizontal arrows are isomorphisms of  Theorem  \ref{t:FR}.
By Corollary  \ref{c:Sansuc-Ch} the right-hand vertical arrow is an isomorphism.
Thus the left-hand vertical arrow is an isomorphism as well, as desired.
\end{proof}

\begin{corollary}[{Sansuc \cite[Corollary 3.5]{Sansuc}}] \
\begin{enumerate}

\item[\rm(i)]  If $G$ splits over a finite Galois extension $L/K$ and all places in $S$ have cyclic decomposition group in $\Gal(L/K)$,
then $G$ has the weak approximation property in $S$.

\item[\rm (ii)] $G$ has the weak approximation property in $\V_\infty(K)$.

\end{enumerate}
\end{corollary}

\begin{theorem}\label{t:AS-H1}
Let $G$ be a reductive group over a global field $K$,
and let $M$, $L$, and $\G$ be as in \ref{ss:MLG}.
Then there is a canonical isomorphism
\[ A_S(G)\isoto\coker\Big[H_1\big(\G,M[\V_L]_0\big)\to \bigoplus_{v\in S} H_1(\G\hm_\bv,M)\Big]\]
where $H_1$ denotes group homology.
\end{theorem}

We need a lemma and a corollary.

\begin{lemma}\label{l:Gt0}
Let $\G$ be a finite group, and let $N$ be a finitely generated {\emm free} $\Z[\G]$-module.
Then
\begin{enumerate}
\item[\rm(i)] $N_\Gt=0$, and
\item[\rm(ii)] $H_i(\G,N)=0$ for all $i\ge 1$.
\end{enumerate}
\end{lemma}

\begin{proof}
By Shapiro's lemma we can reduce our lemma to the case $\G=1$ and $M=\Z$, in which  assertions (i) and (ii) are obvious.
\end{proof}

\begin{corollary}\label{c:Gt0}
Let $L/K$ be a finite Galois extension of global fields with Galois group $\G=\Gal(L/K)$,
and let  $N$ be a finitely generated {\emm free} $\Z[\G]$-module.
Then
\begin{enumerate}
\item[\rm(i)] $\big(N[\V_L]_0\big)_\Gt=0$, and
\item[\rm(ii)] $H_i(\G,N[\V_L]_0)=0$ for all $i\ge 1$.
\end{enumerate}
\end{corollary}

\begin{proof}
By \cite[Theorem A.1.1]{BK}, the short exact sequence
\[ 0\to N[\V_L]_0\to N[\V_L]\to N\to 0\]
gives rise to a homology exact sequence
\begin{align}
\dots\to H_2(\G,N[\V_L]_0)\to H_2(\G,&N[\V_L])\to H_2(\G,N)\notag\\
\to H_1(\G,N[\V_L]_0)\to &H_1(\G,N[\V_L])\to H_1(\G,N)\label{e:exact-N}\\
\to\big(N[\V_L]_0&\big)_\Gt\to \big(N[\V_L]\big)_\Gt\to N_\Gt\notag
\end{align}
By Lemma \ref{l:Gt0} we have  $H_i(\G,N)=0$ for all $i\ge 1$.
Moreover, we have canonical isomorphisms
\begin{equation*}
\begin{aligned}
& H_i(\G,N[\V_L]\cong\bigoplus_{v\in \V_K}\! H_i(\G\hm_\bv,N)=0\ \,\ \text{for all $i\ge 1,$}\\
&\big(N[\V_L]\big)_\Gt\cong\bigoplus_{v\in\V_K}\! N_\Gbvt=0\
\end{aligned}
\end{equation*}
by Shapiro's lemma and   Lemma \ref{l:Gt0} (because the $\Z[\G]$-free module $N$ is $\Z[\G\hm_\bv]$-free).
Now the corollary follows from the exactness of \eqref{e:exact-N}.
\end{proof}

\begin{proof}[Proof of Theorem \ref{t:AS-H1}]
For an $L/K$-free resolution  of Definition \ref{d:L/K-free},  the short exact sequence
of $\Gal(L/K)$-modules \eqref{e:Y-M'-M}
induces a short exact sequence
\[ 0\to \Y[\V_L]_0\to M'[\V_L]_0\to M[\V_L]_0\to 0\]
and a commutative diagram with exact rows
\begin{equation*}
\xymatrix{
0=H_1(\G, M'[\V_L]_0)\ar[r] &H_1(\G, M[\V_L]_0)\ar[r]\ar[d]^{l_1} &\big(\Y[V_L]_0\big)_\Gt\ar[r]\ar[d]^{l_0} &\big(M'[V_L]_0\big)_\Gt=0\\
0=\bigoplus\limits_{v\in S}\! H_1(\G\hm_\bv,M')\ar[r] &\bigoplus\limits_{v\in S}\! H_1(\G\hm_\bv,M)\ar[r]
&\bigoplus\limits_{v\in S}\! \Y_\Gbvt\ar[r] &\bigoplus\limits_{v\in S} (M')_\Gbvt=0
}
\end{equation*}
In this diagram, the zeros in the bottom row are explained by Lemma \ref{l:Gt0},
and the zeros in the top row are explained by Corollary \ref{c:Gt0}.
The diagram induces an isomorphism
$\ \coker l_1\isoto \coker l_0$.\
By Theorem \ref{t:AS-0} we have an isomorphism $A_S(G)\cong \coker l_0$,
which gives a desired isomorphism $A_S(G)\cong \coker l_1$.
\end{proof}

\newcommand{\Cyc}{{\rm Cyc}}

\begin{subsec}
Consider the  sets $S$, $S^\complement$; then $S\cup S^\complement=\V_K$\hs.
Let  $S_\nc\subseteq S$ and $(S^\complement)_\nc\subset S^\complement$
be the subsets consisting of all places $v$ with {\em non-cyclic} decomposition group $\G\hm_w$
where $w$ is a place of $L$ over $v$.
Let $\Cyc(\G)$ denote the set of all {\em cyclic} subgroups $\D\subseteq\G$.

Consider the groups
\[\Xi_S= \bigoplus_{v\in S} H_1(\G\hm_\bv,M)\quad\ \text{and}\quad\ \Xi_{S^\complement}= \bigoplus_{v\in S^\complement}\! H_1(\G\hm_\bv, M).\]
Moreover, consider the groups
\[\Xi_{S_\nc}= \bigoplus_{v\in S_\nc} H_1(\G\hm_\bv,M)\quad\ \text{and}\quad\ \Xi_{(S^\complement)_\nc}= \bigoplus_{v\in (S^\complement)_\nc}\! H_1(\G\hm_\bv, M).\]
Furthermore, consider the group
\[\Xi_\Cyc= \bigoplus_{\D\in\Cyc(\G)} H_1(\D,M).\]

For $v\in\V_K$ consider the natural corestriction homomorphism
\[\tau_v\colon H_1(\G\hm_\bv,M)\to H_1(\G,M).\]
We have a natural homomorphism
\begin{equation*}
\tau_{S\phantom{^\complement}}\colon\, \Xi_{S}\to H_1(\G,M),  \ \,\big[ \xi_v\big]_{v\in S}\mapsto \Bigg[\,\sum_{v\in S_\nc}\tau_v(\xi_v)\,\Bigg].
\end{equation*}
Moreover,  we have a natural homomorphism
\begin{equation*}
\tau_{S^\complement}\!\colon\,  \Xi_{S^\complement}\!\to H_1(\G,M).
\end{equation*}
Furthermore, we have  natural homomorphisms
\[
\tau_{S_\nc\phantom{^\complement}\!\!}\colon\,  \Xi_{S_\nc}\!\to H_1(\G,M),\ \quad
\tau_{(S^\complement)_\nc}\colon\, \Xi_{(S^\complement)_\nc}\!\to H_1(\G,M),\ \quad
\tau_{\Cyc{\phantom{^\complement}}}\!\!\colon\, \Xi_\Cyc\to H_1(\G,M).
\]
\end{subsec}

\begin{theorem}\label{t:AS-H1-cor}
There are  canonical isomorphisms
\[ A_S(G)\,
\cong\,\im \tau_{S_\nc\phantom{^\complement}}\!/
\big(\im\tau_{S_\nc\phantom{^\complement}}\!\!\!\cap\im \tau_{S^\complement}\hs\big)
\,\cong\,\im \tau_{S_\nc\phantom{^\complement}}\!/
\big(\im \tau_{S_\nc\phantom{^\complement}}\!\!\!\cap
(\im \tau_{(S^\complement)_\nc}\! + \im\tau_{\Cyc{\phantom{^\complement}}}\!\!)\hs\big).\]
\end{theorem}

\begin{proof}
For  each $v\in\V_K$, the short exact sequence \eqref{e:Y-M'-M}
gives rise to a commutative diagram with exact rows
\[
\xymatrix@C=12mm{
0=H_1(\G\hm_\bv,M')\ar[r] &H_1(\G\hm_\bv,M)\ar[r]\ar[d]_-{\tau_v} &Y_\Gbvt\ar[d]^-{\pi_v}\ar[r] &(M')_\Gbvt=0 \\
0=H_1(\G,M')\ar[r]       &H_1(\G,M)\ar[r]                                &\Y_\Gt  \ar[r]                      &(M')_\Gt=0
}
\]
In this diagram,  the zeros are explained by Lemma \ref{l:Gt0}.
Now we obtain the first isomorphism of the theorem  from  Theorem \ref{t:AS-0-cor}
giving  isomorphisms
\[A_S(G)\,\isoto\, \Ups_{S_\nc}\hs/\hs \pi_{S_\nc}^{-1}\big(\im \pi_{S^\complement}\big)\, \cong\,
\im \pi_{S_\nc}\hs /\big(\im\pi_{S_\nc}\cap\im \pi_{S^\complement}\hs\big),\]
and from the isomorphisms
\[ H_1(\G\hm_\bv,M)\isoto Y_\Gbvt\hs,\quad\  H_1(\G,M)\isoto\Y_\Gt \]
coming from the  above diagram.
Further, by the Chebotarev density theorem we have
\[ \im\tau_{S^\complement}
=\im \tau_{(S^\complement)_\nc}\! + \im\tau_{\Cyc{\phantom{^\complement}}}\!,\]
which gives the second isomorphism of the theorem.
\end{proof}

\begin{corollary}\label{c:AS-H1-cor}
Assume that there exist a place $v\in S_\nc$ of $K$
and a place $w$ of $L$ over $v$ such that $\G\hm_w=\G$.
Then
\[A_S(G)\cong\coker  \tau_{S^\complement}.\]
\end{corollary}

\begin{proof}
Indeed, then the  homomorphism
$\tau_{S_\nc\phantom{^\complement}\!\!}$ is surjective.
\end{proof}

\begin{corollary}\label{c:AS-H1-cor-bis}
Assume that there exist a place $v\in S^\complement$ of $K$
and a place $w$ of $L$ over $v$ such that $\G\hm_w=\G$.
Then
\[A_S(G)=0.\]
\end{corollary}

\begin{proof}
Indeed, then the  homomorphism $\tau_{S^\complement}$ is surjective.
\end{proof}

\section{An example in positive characteristic}
\label{s:example}

We take
$K=\F_p(t)$, the field of rational functions in one variable $t$ over a finite field $\F_p$ where $p$ is a prime.
We assume that $p\equiv -1\pmod{4}$, for instance, $p=3$ or $p=7$,
and we take $L=\F_p\big(\sqrt{t},\sqrt{t^2-1}\,\big)$.
Then $L/K$ is a Galois extension with Galois group $\G=\Gal(L/K)\simeq \Z/2\Z\times\Z/2\Z$.

\begin{lemma}
\label{l:Sawin}
Let $K =\F_p(t)$, $L= \F_p ( \sqrt{t}, \sqrt{t^2-1} )$ where  $p\equiv -1 \bmod 4$.
Then the decomposition groups for $L/K$ at the places corresponding to the irreducible polynomials  $t$ and $t+1$ coincide with $\G$,
and all the other decomposition groups are cyclic.
\end{lemma}

\begin{proof}(compare Sawin \cite{Sawin} in the case $p\equiv 1\pmod {4}$ ).
To find all places of $K$ with noncyclic decomposition subgroup,
it suffices to consider the  places where at least one of the two  extensions
$\F_p(\sqrt{t})/\F(t)$ and $\F_p(t,\sqrt{t^2-1})/\F(t)$
ramify, because  the unramified extensions of non-archimedean local fields are always cyclic.

The first extension ramifies at $0, \infty$.
At $t=0$, the second extension $y^2=t^2-1 $  locally looks like $y^2 = 0-1 = -1$,
which is a field since $q \equiv -1 \bmod 4$.
Thus $L_w/K_v$ is a composite of a ramified quadratic extension and an unramified quadratic extension
of non-archimedean local fields,
and hence the Galois group $\G\hm_w=\Gal(L_w/K_v)$ is isomorphic to the Klein four-group $\G$. Thus $\G\hm_w=\G$.
 We can write  the second extension as $\left(\frac{y}{t}\right)^2= 1- \frac{1}{t^2}$,
which at $t=\infty$ locally looks like $\left(\frac{y}{t}\right)^2 = 1- \frac{1}{\infty^2}  =1$, which is split.
Thus $L_w/K_v$ is a quadratic extension, and therefore $\G\hm_w=\Gal(L_w/K_v)$ is a cyclic group of order 2.

The second extension ramifies at $t=1$, where the first extension locally looks
like $y^2=1$, hence splits, and therefore $\G\hm_w$ is a cyclic group of order 2.
Moreover, it ramifies at $t=-1$, where the first extension looks like  $y^2=-1$,
hence is a field, because $p\equiv -1\bmod 4$. As above, we see that $\G\hm_w=\G$ in this case.

We see that  $\G\hm_w=\G$ at the places corresponding to the irreducible polynomials  $t$ and $t+1$,
and $G_w$ is cyclic for all other places $v $ of $K$, as desired.
\end{proof}

\begin{proposition}
Let  $T=R^1_{L/K}\Gm\coloneqq \ker\big[R_{L/K}\Gm\to\Gm\big]$ be the norm 1 torus,
where $L/K=\F_p\big(\sqrt{t},\sqrt{t^2-1}\,\big)/\F_p(t)$ with $p\equiv -1\pmod{4}$.
Let $S$ be a finite set of places of $K$.
 If $S$ contains both places of $K=\F_p(t)$ corresponding to the polynomials
$t$ and $t+1$, then $A_S(T)\cong \Z/2\Z$.
Otherwise we have $A_S(T)=0$.
\end{proposition}

\begin{proof}
This is similar to \cite[Example 5.6]{Sansuc} in the number field case.
Consider the  algebraic fundamental group
\[M=\pi_1^\alg(T)=\X_*(T)=\ker\big[\hs\Z[\G]\to\Z\hs\big].\]
We wish to compute $H^{-2}(\G,M)$ and $H^{-2}(\G\hm_w,M)$.
From the short exact sequence
\[0\to M\to \Z[\G]\to\Z\to 0\]
we obtain an exact sequence of Tate cohomology groups
\[0=H^{-3}(\G,\Z[G])\to H^{-3}(\G,\Z)\to H^{-2}(\G,M)\to H^{-2}(\G,\Z[\G])=0.\]
Thus
\[H^{-2}(\G,M)\cong H^{-3}(\G,\Z)=H_2(\G,\Z).\]
Similarly we obtain that
\[H^{-2}(\G\hm_w,M)\cong H^{-3}(\G\hm_w,\Z)=H_2(\G\hm_w,\Z).\]
Now  $H_2(\G,\Z)$ is the Schur multiplier of $\G$, and by a theorem of Schur we have
\[H_2(\G,\Z)\cong\Z/2\Z;\]
see \cite[Corollary 2.2.12]{Karpilovsky}.
Thus $H^{-2}(\G,M)\cong\Z/2\Z$.

If $\G\hm_w=\G$, then, of course,
\[H^{-2}(\G\hm_w,M)=H^{-2}(\G,M)\cong\Z/2\Z.\]
If $\G\hm_w$ is cyclic, then we have
\[H^{-2}(\G\hm_w,M)\cong H^{-3}(\G\hm_w,\Z)\cong H^1(\G\hm_w,\Z)=\Hom(\G\hm_w,\Z)=0,\]
where $H^{-3}(\G\hm_w,\Z)\cong H^1(\G\hm_w,\Z)$ by periodicity; see \cite[Section IV.8, Theorem 5]{CF}.

Now the proposition follows  from  Corollaries \ref{c:AS-H1-cor} and \ref{c:AS-H1-cor-bis}.
\end{proof}

\appendix

\section{Existence of an $L/K$-free resolution of a reductive group}
\label{app:CT}

\begin{center}
{\em Jean-Louis Colliot-Th\'el\`ene}
\end{center}
\medskip

In this appendix we give an alternative proof of Proposition \ref{p:E-L/K-free}.

\begin{proof}
Let $G$ be a reductive group over a field $K$. Assume that there exists a maximal $K$-torus $T\subset G$
which is split by a Galois extension $L/K$ with Galois group $\Gamma$.
This implies that $G^{\text{tor}}$ and the centre of $G^{\text{sc}}$, as groups of multiplicative type, are split by $L/K$.
From this, according to Remark 3.1.1 in \cite{CT-RF}, if one follows the proof of Proposition-D\'efinition 3.1, then, given $G$, one may find a flasque resolution
\[
1 \to S \to H \to G \to 1
\]
such that the $K$-tori $S$ and $P$ are split by $L$.

We now follow the notation in the proof of Proposition-D\'efinition 3.1 in \cite{CT-RF}
and we use the diagram constructed on page 89.
The $K$-torus $Z$ in this proof is split by $L/K$.
In this proof, one may choose for $Q$  a $K$-torus split by $L/K$,
with character group a free $\mathbb{Z}[\Gamma]$-module.

Let $G^{\prime} := G^{\text{sc}} \times Q$. This is an $L/K$-free group; see Definition \ref{d:L/K-free-group}.

As on  page 89 of \cite{CT-RF},
consider the exact sequence
\begin{equation}\label{e:1}
1 \to G^{\text{sc}} \times Q \to H \to P \to 1,
\end{equation}
that is,
\[
1 \to G^{\prime} \to H \to P \to 1.
\]
Let
\begin{equation}\label{e:2}
1 \to B \to R \to P \to 1
\end{equation}
be an exact sequence of $K$-tori split by $L/K$ such that
the character group $\X^*(R)$ of $R$ is a free $\mathbb{Z}[\Gamma]$-module.

Now pull back \eqref{e:1} via \eqref{e:2}.
One gets an exact sequence
\begin{equation}\label{e:3}
1 \to G^{\prime} \to H_{1} \to R \to 1
\end{equation}
and an exact sequence
\[
1 \to B \to H_{1} \to H \to 1.
\]
The kernel of the composite map $H_{1} \to H \to G$ is an extension of the (flasque) $K$-torus $S$ by the $K$-torus $B$, hence is a $K$-torus $C$ split by $L/K$.

One then has the exact sequence
\begin{equation}\label{e:4}
1 \to C \to H_{1} \to G \to 1
\end{equation}
with $C$ a $K$-torus split by $L/K$, and the exact sequence
\[
1 \to G^{\prime} \to H_{1} \to R \to 1,
\]
that is,
\[
1 \to (G^{\text{sc}} \times Q) \to H_{1} \to R \to 1
\]
with $\X^*(R)$ a free $\mathbb{Z}[\Gamma]$-module. One then has an exact sequence
\[
1 \to G^{\text{sc}} \to H_{1} \to M \to 1
\]
where $M$ is an extension of $R$ by $Q$ and hence is a $K$-torus split by $L/K$ whose character group is a free $\mathbb{Z}[\Gamma]$-module. So $H_{1}$ is an $L/K$-free group, see Definition \ref{d:L/K-free-group},
and \eqref{e:4} is a desired $L/K$-free resolution of $G$, see Definition \ref{d:L/K-free}.
\end{proof}


\section{A Magma program}
\label{app:Magma}

In the preprint version, we provide a computer program computing $A_S(G)$
assuming that we know the effective Galois  group $\G$ acting on $M=\pia(G)$,
all {\em non-cyclic}  decomposition groups (up to conjugacy) $\G\hm_w\subset \G$
for  $v\in S$ and $w\in\V_L$ over $v$,
and all {\em non-cyclic} decomposition groups $\G\hm_w$ for $v\in S^\complement$.
We use Theorem \ref{t:AS-H1-cor}.
\bigskip\bigskip

\begin{lstlisting}
/*  This program computes the defect 
of weak approximation A_S(\GG)
where \GG is a connected reductive group 
over a global field K.
We assume that we know:

*  the effective Galois group G=Gal(L/K) ={1,2,\dots, gg}
   of order gg with multiplication law given 
   by a martrix mu of dimensions gg x gg ;

*  the algebraic fundamental group \pi_1(\GG) given by Z^mm / <N>
   where N is a nn * mm -matrix, with the action of G on \pi(G) 
   given by gg matrices mm x mm ;

*  the set of all non-cyclic (and maybe some cyclic) 
   decomposition subgroups  G_v of G for v in S ;

*  the set of all non-cyclic (and maybe some cyclic) 
   decomposition subgroups G_v of G 
   for v in the commplement S^C of C.

We useTheorem 4.10. 
*/



ComplementIntMat:= function( full, sub )
/* This is a Magma version of the GAP function 
with the same name.
This function was written by Willem A. de Graaf.*/

     full:= Matrix( Integers(), full );
     H:= HermiteForm( full );
     full:= [ ];
     for i in [1..NumberOfRows(H)] do 
         if not IsZero( H[i] ) 
			then Append( ~full, Eltseq( H[i] ) ); 
	 	end if;
     end for;
     full:= Matrix( Integers(), full );

     n:= NumberOfColumns(full);
     m:= NumberOfRows( full );
     Vn:=KModule( Rationals(), n );
     W:= KModuleWithBasis( [ Vn!Eltseq(full[i]) : i in [1..m] ] );
     cfs:= [ Coordinates( W, Vn!Eltseq(sub[i]) ) : 
	 i in [1..NumberOfRows(sub)] ];
     A:= Matrix( Integers(), cfs );
    
     S,P,Q:= SmithForm( A );
     Qi:= Q^-1;
     cfs:= [ Eltseq(Qi[j]) : j in [1..m] ];
     vecs:= [ &+[ c[i]*(Vn!Eltseq(full[i])) : 
	     i in [1..m] ] : c in cfs ];
     x:= [ ];
     for i in [1..NumberOfColumns(S)] do 
         if S[i][i] ne 0 then 
            Append( ~x, S[i][i] );
         end if;
     end for;
     
     comp:= [ ];
     for i in [#x+1..#vecs] do Append( ~comp, vecs[i] ); end for;
     
     return Matrix( Integers(), comp), Matrix( Integers(),  
	 	[ Eltseq(vecs[i]) : i in [1..#x] ] ), 
	    Vector( Integers(), x );
     
end function;




//CHECK FOR mu 
IsGroup := function(gg, mu)	
/* This fuction checks whether the matrix mu 
defines a group structure on the set {1,2,...,gg},
and if yes, it computed the vector iota 
describing the inversion map.*/
	
	boo:=true;	
	
	for g:=1 to gg do
		boo := boo and (mu[g,1] eq g) and (mu[1,g] eq g);
//      Checking the unit condition
	end for;	
				
	for g1:=1 to gg do
	for g2:=1 to gg do
		for g3:=1 to gg do
			p12:=mu[g1,g2]; p23:=mu[g2,g3];
			boo:=boo and (mu[p12,g3] eq mu[g1,p23]);
//			Checking the associativity condition
		end for;
	end for;
	end for;
		
	iota:=[];

	for g:=1 to gg do
		boog:=false;
		for h:=1 to gg do
			boo1:=(mu[g][h] eq 1);
			if boo1 then
		    	iota[g]:=h;
				boog:=true;
				break;
			end if;
		end for;
		boo:=boo and boog;
//	Checking the existence of an inverse element
	end for;
		
	return boo,iota;
	
end function;




// CHECK FOR N
IsAction:= function(gg,mu,mm,rho,nn,N)	
	
/* This function checks whether rho defines an action 
of G on Z^mm/<N>. */
	
	boo:=true;	
	Zv:=RModule(Integers(), mm);
	v0:= Zero(Zv);
	vv:= Zero(Zv);			

	for g:=1 to gg do
	for i:=1 to nn do
		for j:=1 to mm do
			vv[j]:=N[i][j];
		end for;
		vv:=vv * rho[g];	
		ic:=IsConsistent(N, vv); 
		/*Checking that rho preserves  the subgroup
                     <N> of M=Z^mm 
		generated by the rows of the matrix N */		
		boo:=boo and ic;
	end for;
	end for;
	
	for g:=1 to gg do
	for h:= 1 to gg do
		Dif:=rho[mu[g][h]] - rho[g]*rho[h];
		for i:=1 to mm do
			for j:=1 to mm do
				vv[j]:=Dif[i][j];
			end for;	
			ic:=IsConsistent(N, vv); //print ic; 
// 			Checking that rho is a well-defined 
//			action modulo <N>
			boo:=boo and ic;
		end for;
	end for;
	end for;
		
	return boo;
	
end function;
			




NullModN:= function(A,N) 
	AN:=VerticalJoin(A,N);
	K:=KernelMatrix(AN);
	rA:=NumberOfRows(A);
	rK:=NumberOfRows(K);
	B:=ExtractBlockRange(K, 1, 1, rK,rA);
	Meet:=B*A;
	return B, Meet;
//  <B> is the preimage of <N> under A,
// <Meet> is the intersection of <A> and <N>.	
	
end function;


	


Minus2CoCycles:= function(gg, mu, iota, mm, rho, nn, N)	
		
//COMPUTING THE MINUS 2 COCYCLES Z^{-2}(G, Z^mm / <N> )
		
	matz:=ZeroMatrix(Integers(), mm, mm );
		
	Idm:=ScalarMatrix(Integers(), mm, 1);

	for h:=1 to gg do
		matz:=rho[h]-Idm;
		if h eq 1 then 
			Dm1:=matz; 
		else 
			Dm1:=VerticalJoin(Dm1,matz);
		end if;
	end for; 
		
	Zm2:=NullModN(Dm1,N);	
			
	return Zm2;
	//  <Zm2> is our group of -2 cocycles 

end function;





Minus2CoBoundaries:= function(gg, mu, iota, mm, rho, nn, N)	
		
//COMPUTING THE MINUS 2 COBOUNDARIES B^{-2}(G, Z^mm / <N>)
		
	Zv:=RModule(Integers(), mm);
	v0:= Zero(Zv);
		
	Idm:=ScalarMatrix(mm,1);
		
	Bm2:=ZeroMatrix(Integers(), gg*gg*mm, gg*mm );
		
	c0:=[ []:i in [1..gg]];
	for g in [1..gg] do
	for h in [1..gg] do
		c0[g,h]:=v0;
	end for;
	end for;

	matz:= [Matrix(Integers(), Idm) : u in [1..mm] ];
	
//	We have a triple cycle
	for g0:=1 to gg do
	for h0:=1 to gg do
	for m0:=1 to mm do
		c:=c0;
		c[g0][h0][m0]:=1;

/*	We consider a 3-dimensional zero array c 
of dimensions gg x gg x mm. 
All values c[g,h,m] are zero except for 
c[g0,h0,m0] that equals 1.
	
We obtain gg * gg * mm vectors c[g,h] 
in our module M of rank mm.
Then we compute the boundary dc 
by the following double loop:
*/
		
		for g:=1 to gg do
			
			sum:=v0;
			for h:=1 to gg do	
				sum:=sum + c[h,g]*rho[h] 
				- c[h, mu[iota[h],g]] + c[g,h];	
			end for;
				
			i1:=m0+mm*(h0-1)+gg*mm*(g0-1);
			for m:=1 to mm do
				j1:=mm*(g-1)+m;
				Bm2[i1][j1]:=sum[m];
			end for;
		end for;
		
	end for;
	end for;
	end for;
		
		
	for g:= 1 to gg do
		if g eq 1 then
			ggN:=N;
		else
			ggN:=DiagonalJoin(ggN,N);
		end if;
	end for;

	Bm2:=VerticalJoin(Bm2,ggN);
				
// <Bm2> is our group of -2 coboundaries modulo <N> 
		
	return ggN, Bm2;
		
end function;
	



Zm2FromCyclic:= function(gg, mu, iota, mm, rho, nn, N)
	
/* 
We compute  the images of (-2)-cocycles for 
the cyclic subgroups <g> for all g in G. 
To g in G and m in M=Z^mm such that m is g-fixed modulo <N>,
we assign the (-2)-cocycle c with c(g)=m
and c(h)=0 for all h in G different from g.
Taking for m a basis element of Fixg, 
we obtain the image of Z^{-2}(<g>,M)in Z^{-2}(G,M).
*/			
			
	Idm:=ScalarMatrix(Integers(),mm,1);

	for g:= 1 to gg do
			
		Fixg,Meet:=NullModN(Idm-rho[g],N);	
/* Fixg are the elements of M=Z^mm that are fixed 
by iota[g] modulo <N> */
	
		if g eq 1 then
			ALDiag:=Fixg;
		else
			ALDiag:=DiagonalJoin(ALDiag,Fixg);
		end if;
			
	end for;
		
	return ALDiag;
		
end function;	




AssignmentsH:= function(gg,mu,iota,mm,rho, hh,phiH)
	
// phiH is the embedding H={1..hh} into G={1..gg}
// we compute the restrictions to H of mu, iota, and rho.
	
	muH:=ZeroMatrix(Integers(),hh,hh);
	iotaH:=[];
	psiH:=[];
	
	for i:=1 to gg do
		psiH[i]:=0;
		for j:=1 to hh do
			if i eq phiH[j] then
				psiH[i] := j;
				break;
			end if;
		end for;
	end for;	
/* psiH is the inverse function for phiH;
it takes value 0 where it is not defined.
*/
				
		
	for g:=1 to hh do
	for h:=1 to hh do
		i:=mu[phiH[g],phiH[h]];
		muH[g][h]:=psiH[i];
		if psiH[i] eq 0 then
			 print "\nphiH is wrong for mu";
		end if;	
	end for;
	end for; 
/* muH is the restriction to H  
of the multiplication law mu in G */
		
	for h:=1 to hh do
		i:=iota[phiH[h]];
		iotaH[h]:=psiH[i];
		if psiH[i] eq 0 then
			print "\nphiH is wrong for iota";
		end if;	
	end for;  
/* iotaH is the restriction to H 
 of the inversion map iota for G */
	print "\n**************************";
	print "\nphiH", phiH;	
	print "\nmuH", muH;
	print "\niotaH",iotaH;
	
	rhoH:=[ Matrix(Integers(),rho[phiH[u]]) : u in [1..hh] ];
	print "\nrhoH", rhoH; 
	//rhoH is the restriction to H of the action rho of G
	
	return psiH, muH, iotaH, rhoH;
	
end function;		
		
		

Xi:= function(gg,mu,iota,mm,rho,nn,N,phi)
	
	/* Computing Xi_S= \bigoplus_{i in S} Zm2 
	which is the direct sum of Zm2
	over the set S of subgroups given by phi S given by phi
	*/
		
	nD:=#(phi);
	AL:=ZeroMatrix(Integers(), 1,gg*mm);
	if nD eq 0 then
		return AL;
	else
	
	for l:=1 to nD do
	
		hh:=#(phi[l]);
		phiH:=phi[l];
	
		psiH, muH, iotaH, rhoH 
		:= AssignmentsH(gg,mu,iota,mm,rho, hh,phiH);
	
		Z2H:=Minus2CoCycles(hh,muH,iotaH,mm,rhoH,nn,N);
		rZ2H:=NumberOfRows(Z2H);
	
		A:=ZeroMatrix(Integers(), rZ2H, gg*mm);
	
		for i:=1 to rZ2H do
		for h:=1 to hh do
		for j:=1 to mm do
			A[i][(phiH[h]-1)*mm+j] 
			:= Z2H[i][(h-1)*mm+j];
			// We embed each 1-cycle of H into G
		end for;
		end for;
		end for;
	
		AL:=VerticalJoin(AL,A);
	
	end for;
	end if;
	
	return AL; 
/* AL is the list of all 1-cycles of of G 
coming from 1-cycles of H 
for all H in S, where S and the subgroups H in S 
(subgroups of G) are given by phi
*/
	
end function;

// WE START OUR COMPUTATION 	
	

print "\n\nInput";

	gg:=4; 	
	mu:=[[1,2,3,4],[2,1,4,3],[3,4,1,2],[4,3,2,1]];		
	// Klein's four-group
		
	mm:=3;
	Idm:=ScalarMatrix(Integers(), mm,1);
	rho:= [ Matrix(Integers(),Idm) : u in [1..gg] ];
	
	if mm eq 1 then 
			rho[2]:=[[3]];
			rho[3]:=[[5]];
			rho[4]:=[[7]];
	end if;
	
	
	if mm eq 4 then
		rho:= [ Matrix(Integers(),Idm) : 
		u in [1..gg] ];

		for g:=1 to gg do
	
			rhr:=ZeroMatrix(Integers(), gg, gg);
			for h:=1 to gg do
				rhr[h,mu[g,h]]:=1;
			end for;
		
			rho[g]:=rhr;
				 	
		end for;
	end if;
	
		
	if mm eq 3 then
		for g:=1 to gg do
	
			rhr:=ZeroMatrix(Integers(), gg, gg);

			for h:=1 to gg do
				rhr[h,mu[g,h]]:=1;
			end for;
			
			U:=Matrix(RationalField(), 4,4, 
			[1,0,0,-1,  0,1,0,-1,  0,0,1,-1, 1,1,1,1]);	
			rhr:=U*rhr*U^(-1);
			rho[g]:=ExtractBlock(rhr,1,1,mm,mm);
	 	
		end for;
	end if;
		 
		 		

		nn:=3;
	
	if nn eq 0 then 
		N:=ZeroMatrix(Integers(), 1, mm);
	end if;
	
	if nn eq 1 then
		N:=ZeroMatrix(Integers(), nn,mm);
		for m:=1 to mm do
			N[1,m]:=1;
		end for;
	end if;
	
	if nn eq 3 then
		N:=4*ScalarMatrix(Integers(),mm, 1); // Modulo 4
	end if;
	
	if nn eq 5 then
		N:= ZeroMatrix(Integers(), mm+1,mm);
		for m:=1 to mm do
			N[m,m]:=1;
			N[mm+1,m]:=1;
		end for;
	end if;
	
	
//	We take: 
	
phiNCyc:=[[1,2,3,4]];	

phiC:=[[1,2]];
phiS:=[[1,2,3,4],[1,3]];



print "\nGamma=[1..gg],  gg = ", gg; 
print "\nThe multiplication table: \nmu ", mu;	
	
print "\nM=Z^mm,  mm=", mm; 
print "\nThe action of G on M/N ( gg matrices  mm x mm ):";
print "\nrho ", rho;	

print "\nThe number of relations in M:  \nnn=", nn; 
"\nThe relations in Z^mm:  <N>, \nN= ",N;

print "\nOur G-module is Z^mm / <N> ";	
	

boo, iota := IsGroup(gg,mu);	
print "\nCheck: Is G a group?  ", boo;

print "\nThe inversion table: iota ", iota;		
print "\nCheck:  Is rho an action of G on M/N?  ", 
	IsAction(gg,mu,mm,rho,nn, N);



print "\n\nComputing the Tate cohomology H^{-2}(G,M/N)";				
		
Zm2:=Minus2CoCycles(gg, mu, iota, mm, rho, nn, N);		
ggN,Bm2:=Minus2CoBoundaries(gg, mu, iota, mm, rho, nn, N);	
// ggN is DiagonalJoin N gg times.
		

comp, bas,x:=ComplementIntMat(Zm2,Bm2);
print "\nH^{-2}(G,M/N) = \prod_i Z/x[i]Z where ";
print "x = ", x;	
	


print "\n\nComputing the defect of weak approximation A_S 
corresponding to all cyclic decomposition subgroups, 
    a list of all non-cyclic  (and maybe some cyclic) 
decomposition subgroups in the complement  S^C of S,
    and a list of non-cyclic  (and maybe some cyclic) 
decomposition subgroups in S";


print "\nThe decomposition subgroups in S^C: \nphiC:", phiC;
print "\nThe decomposition subgroups in S: \nphiS:", phiS; 

print "Below phiH is the emgedding 
 of H=[1..hh] into Gamma=[1..gg].";
print "We compute the restrictions to H of mu, iota, and rho.";

ALCyclic:=Zm2FromCyclic(gg, mu, iota, mm, rho, nn, N);
AL:=VerticalJoin(Bm2,ALCyclic);			
ALC:= Xi(gg,mu,iota,mm,rho,nn,N,phiC);
ALCJ:=VerticalJoin(ALC,AL);
/* ALCJ is the vertical Join of Zm2 from all cyclic subgroups, 
 Bm2 (all coboundaries and ggN), 
 and (-2)-Cocycles for *non-cyclic* decomposition subgroups 
 in the *complement* S^C of the set S. 
 Here ggN is the Diagonal Join N gg times.
 ALCJ corresponds to 
 im tau_{(S^C)_nc} + im tau_Cyc   
 in Theorem 4.10.
*/

ALS:= Xi(gg,mu,iota,mm,rho,nn,N,phiS);
ALSJ:=VerticalJoin(ALS,Bm2);
/* ALSJ is the vertical Join of Zm2 from all 
 *non-cyclic* decomposition subgroups in the set S 
 and of Bm2 (all coboundaries and ggN). 
 ALSJ corresponds to im tau_{S_nc} in Theorem 4.10.
*/

print "\n**************************";

XX,Meet:=NullModN(ALSJ,ALCJ);
// <Meet> = <ALSJ> \cap <ALCJ>

comp, bas, x := ComplementIntMat(ALSJ,Meet);
/* x corresponds to the set of elements of finite order in 
<ALSJ> / ( <ALSJ> cap <ALCJ> ).
Note that all these elements are of finite order.*/

print "\nThe defect of weak approximation:";
print "A_S(G) = prod_i Z/x[i]Z where ";
print "x = ", x, "\n";
\end{lstlisting}


\def\sk{\smallskip }

\end{document}